\documentclass[11pt]{article}
\usepackage[english]{babel}
\usepackage{enumitem}
\usepackage{amsmath}
\usepackage{amssymb}
\usepackage{amsthm}
\usepackage{latexsym}
\usepackage{color}
\usepackage{graphicx}
\usepackage{color}
\usepackage[pdftex=true, colorlinks=true,linkcolor=red,citecolor=blue]{hyperref} 
\DeclareSymbolFont{calletters}{OMS}{cmsy}{m}{n}
\DeclareSymbolFontAlphabet{\mathcal}{calletters}
\newlength{\oldparindent}
\newtheorem{Theorem}{Theorem}[part]
\newtheorem{Definition}{Definition}[part]
\newtheorem{Proposition}{Proposition}[part]
\newtheorem{Assumption}{Assumption}[part]
\newtheorem{Lemma}{Lemma}[part]

\newtheorem{Remark}{Remark}[part]
\newtheorem{Example}{Example}[part]

\newcommand{\nc}{\newcommand}
\nc{\esssup}{\mathop{\mathrm{ess\,sup}}}
\nc{\essinf}{\mathop{\mathrm{ess\,inf}}}
\nc{\argmax}{\mathop{\mathrm{arg\,max}}}

\def \P{\mathbb{P}}
\def \N{\mathbb{N}}

\def \R{\mathbb{R}}

\def \E{\mathbb{E}}

\def \Q{\mathbb{Q}}

\def \H{\mathbb{H}}
\def \1{\mathds{1}}

\def \l({{\left (}}
\def \r){{\right )}}
\def \l[{{\left [}}
\def \r]({{\right ]}}

\newcommand{\MBFigure}[6]{
$\left. \right.$ \\
\refstepcounter{figure}
\addcontentsline{lof}{figure}{\numberline{\thefigure}{\ignorespaces #5}}
\begin{center}
\begin{minipage}{#1cm}
\centerline{\includegraphics[width=#2cm,angle=#3]{#4}}
\begin{center}
\upshape{F\textsc{ig} \normal
\end{center}
size{\thefigure}. $-$} #5
\end{center}
\label{#6}
\end{minipage}
\end{center}
$\left. \right.$ \\}
\title{Exponential Quadratic BSDEs with infinite activity Jumps
}
 \author{Anis {\sc Matoussi} \footnote{Research supported by the Chair {\it Financial Risks} of the {\it Risk Foundation} sponsored by Soci\'et\'e G\'en\'erale, the Chair {\it Derivatives of the Future} sponsored by the {F\'ed\'eration Bancaire Fran\c{c}aise}, and the Chair {\it Finance and Sustainable Development} sponsored by EDF and Calyon.}\\
Le Mans University \\  Risk and Insurance Institute of Le Mans \\ Laboratoire  Manceau de Math\'ematiques \\ e-mail:  \textcolor[rgb]{0.00,0.07,1.00}{anis.matoussi@univ-lemans.fr.}
			\and 
Rym  {\sc Salhi}\\ Le Mans University\\  Risk and Insurance Institute of Le Mans \\Laboratoire  Manceau de Math\'ematiques\\ e-mail: \textcolor[rgb]{0.00,0.07,1.00}{rym.salhi@univ-lemans.fr}
}
 
\begin{document}
\maketitle
\begin{abstract}
In this paper, we study a Backward Stochastic Differential Equation with Jumps (BSDEJs in short) where the jumps have infinite activity. Following a forward approach based on Exponential Quadratic semimartingale, we prove the existence of solution of Quadratic BSDEJs with unbounded terminal condition and quadratic growth in $z$.
\end{abstract}

\section{Introduction}
In the last two decades, Backward Stochastic Differential Equations (BSDEs in short) was the main tool to solve   stochastic control problems in financial mathematics, for instance, maximization utility problem, robust maximization problem and stochastic differential games. These equations were first introduced by Bismut in 1973 \cite{Bismut} in the context of stochastic optimal control problem. \\ In 1990, Pardoux and Peng \cite{PP90} proved the well posedness for bounded BSDEs with drivers that satisfy a general non-linear Lipschitz condition.\\
 In the special context of recursive utility, Duffie and Epstein  \cite{DE92} also introduced these equations. \\
\\
Since then BSDE's have been widely studied and a particular class of BSDE received a growing interest is the case when the driver has quadratic growth in $z$. The first paper where such BSDE appeared is due to Schroder and Skiadas \cite{Schroder} and followed by Kobylanski who treats the question of existence and uniqueness of BSDE's solution when the terminal condition is bounded. \\
To show that, the author use a monotone stability approach based on an exponential change of variable, truncature procedure and a comparison result. The major difficulty in the so-called Kobylanski method is the strong convergence of the martingale part.\\
Afterward, Tevzadze \cite{T08} provides a totally different approach to solves a locally Lipschitz-quadratic BSDE with bounded terminal condition based on a fixed point argument. The particularity of this methodology is that the strong convergence is no longer needed to get the existence of the solution. However, the argument of this method stands only for a small bounded terminal condition.  This work has been generalized by \cite{KTP} to the finite jump setting.\\
Recently, Barrieu and El karoui \cite{BElK11} proposed a different method to tackle the question of existence of the solution for unbounded quadratic BSDE. The approach is based essentially on the stability of a special class of semimartingale.\\
Many works have also been regarding to the boundedness of the terminal condition. In \cite{BH05},\cite{BH07}, Briand and Hu extended the result of Lepeltier and San Martin \cite{LepSan98} to the case of unbounded terminal condition and provide an existence result for quadratic BSDE.\\
\textbf{ Literature review on BSDEs with Jumps }
Backward stochastic differential equation with jumps (BSDEJs in short) was first introduced by Tang and Li \cite{TangandLi} as follows
\begin{equation}\label{JBSDEE}
 Y_{t}=\int_{t}^{T}f(s,Y_{s^{-}},Z_{s},U_{s})ds-\int_{t}^{T}Z_{s}dW_{s}-\int_{t}^{T}\int_{E}U_{s}(e)\tilde{\mu}(dt,de)\,\,\,\,\,\,\,\,\,\,\,\,\,\,\,t\in[0,T]
\end{equation}
Where $\tilde{\mu}=\mu(dt,de)-\nu(dt,de)$ is the compensated random measure of $\mu$.
Under a Lipschitz condition on the generator $f$, Barles, Buckdahn and Pardoux \cite{BPP97} investigate the well posedness of this equation, in order to give a probabilistic interpretation of viscosity solution of semilinear integral Partial equations.
According to the authors, a solution of BSDEJ associated to $(f,\xi)$ is a triple of progressively measurable processes $(Y,Z,U)$ such that it verifies the equation \eqref{JBSDEE}.\\
Later, Becherer studied separately in the finite and infinite activity setting \cite{B06, B14}  the existence and uniqueness of BSDEJ's solution when the driver is Lipschitz in $(y,z)$ and locally in $u$.  \\
Many attempts have been suggested to relax the assumptions on the driver $f$. However, only a few works studied the quadratic case in a general setting. In fact most of the works are arising from utility maximization problem and hence deals wih the wellposedness of BSDEJ with a specific form of $f$.\\ 
The only general result in this subject is the paper of El Karoui, Matoussi and Ngoupeyou \cite{ELKMN16}. The authors extend the approach of \cite{BElK11}  to the jumps setting. They characterize the Quadratic BSDEJs with a class of quadratic semimartingale called \emph{Exponential Quadratic semimartingale} and provide a stability result for this class of  semimartingale. To get the existence of quadratic BSDE's solution, they use a regularization procedure and the stability result of this semimartingale under minimal integrability assumptions.
Nonetheless, we have to point out that the extension of the forward approach to the jumps setting is not straightforward . The presence of the jumps induce many technical difficulties and requires some specific arguments.\\
In \cite{Fuji13}, Fujii and Takahashi proved the existence and uniqueness of the solution under the exponential quadratic structure of \cite{ELKMN16} and bounded terminal condition. Later, Laeven and Stadje \cite{Laeven} investigate the effect of ambiguity on a portfolio choice and indifference  valuation problem in term of solutions of bounded BSDEJ whose generator grows at most quadratically with infinite activity jumps.\\
 The reason why quadratic BSDE has attracted much attention is the range of applications notably mathematical finance. In \cite{rouge}, Rouge and  El Karoui solved an indifference pricing problem with exponential utility via a BSDE approach. Later, Imkeller and Muller \cite{HuImkellerMueller05} extended the result of Rouge and EL Karoui to the cases of power and logarithmic utility in which the set of strategies is a closed set and related it to a specific class of quadratic BSDE. In \cite{bordigoni}, Bordigoni, Matoussi, and Schweizer studied a robust indifference pricing problem in general setting using stochastic control technics. They have proved that the value function of the stochastic control problem is characterized by the solution of a general quadratic BSDE using the dynamic programming Bellman principle.  \\ 
 In the context of exponential utility maximization problem, Morlais \cite{M07, M10} have shown the existence and the uniqueness of the solution of bounded BSDE with jumps for a specific form of generator.\\
\,\,\textbf{ Main Contribution }
Our main interest is to study quadratic BSDE with infinite activity jumps and unbounded terminal condition when the driver satisfies the following structure condition.
\begin{equation}\label{BSDE}
\underline{q}(t,y,z,u)=-{1\over \delta}j_t(-\delta u)-{\delta\over 2}|z|^2-l_t-c_t|y|\le f(t,y,z,u)\le {1\over \delta}j_t(\delta u)+{\delta\over 2}|z|^2+l_t+c_t|y|=\bar q(t,y,z,u).\nonumber
\end{equation}\\
with $l$, $c$ a non negative processes and $\delta>0$ a constant where $$j_t(\delta u):=\int_{E}(e^{\delta u}-\delta u-1)\nu(de).$$
 Our point of view is inspired from \cite{ELKMN16} who studied the problem in the finite activity jumps setting. The forward point of view is based on characterizing the solution of the quadratic BSDEJ as a special quadratic exponential semimartingale. 
  Since we are dealing with a BSDEs with infinite activity, our approach is based on a truncation technique of the measure $\nu$ combinated with double approximation of the exponential quadratic generator using the inf and sup convolution. This regularization procedure transforms the original quadratic generator with infinite jumps into a sequence of coefficients with Lipschitz growth and finite activity jumps. Applying a general stability result based on an old theorem of Barlow and Protter \cite{BP90}, we get the existence of BSDE's solution with infinite activity jumps. \\
  
The paper is organized as follows. In section $2$ we give preliminaries including all notations and framework. In section $3$ we give the definition of the solution of Quadratic BSDEs with jumps, the main assumptions and we provide some technical lemma and finally we prove and existence result for unbounded solution of quadratic BSDEJs. Finally, we give some technical results needed for the existence of the solution of the BSDEJs.
\section{Preliminaries and main results}
\subsection{Notations and setting }
Recall that $T>0$ is a fixed time horizon. We consider a filtered probability space $(\Omega ,\mathcal{F},\mathbb{P})$ on which is defined a $d$-dimensional Brownian motion $ W=(W_{t})_{0\leq t\leq T}$ and an integer valued random measure $\mu$ with a compensator
that can be time-inhomogeneous  and may allow for infinite activity of the jumps. We assume that the filtration $\mathbb{F}=(\mathcal{F}_{t})_{0<t<T}$ satisfies the usual conditions of completeness and right continuity. Due to these usual conditions, we can take all semimartingales having right continuous paths with left limits.
Here $\mu$ is an integer valued random measure defined as
\begin{align}
\mu(\omega,dt,de): &(\Omega\times[0,T]\times\mathcal{E})\rightarrow(\mathcal{B}([0,T])\times \mathcal{E})  \nonumber \\
                  &(w,dt,de)\rightarrow \mu(\omega,dt,de)=\sum^{\Delta X_{s}\neq 0}_{s\in[0,t]}\delta_{ (s,\Delta X_{s}) }(dt,de).\nonumber
\end{align}
We will denote by $\nu$ the compensator of $\mu$ under the probability measure $\mathbb{P}$. For a $\sigma$-finite measure $\lambda$ and a bounded Radon-Nikodym derivative $\zeta$, we will assume that the compensator $\nu$ is absolutely continuous with respect to $\lambda\otimes dt$ such that for some constant $C_{\nu}\geq 0$, 
$$ \nu(dt,de)=\zeta(\omega,t,e)\lambda (de)dt.\hspace{1cm} 0 \leq \zeta(t,e)\leq C_{\nu} $$
and  satisfies the following integrability condition $\int_{E}(1\wedge |e|^{2})\lambda (de)<\infty $.\\
In this work, we specially pay attention to the case when the jumps have infinite activity meaning that $\lambda(E)=\infty$. Note that in \cite{ELKMN16} the case of finite activity is already  considered . \\ As stated in \cite{sato}, the infinite activity of the jumps is related to the behavior of the compensator $\nu$ near to $0$. Since we have always a finite number of big jumps, the (in)finitness of the  jumps is controlled by the number of small jumps and thus the behavior of the $\nu$ around the origin.

Let $f$ be a $\mathcal{P}\otimes\mathcal{E}$-measurable function. The integral with respect to the random measure and the compensator are defined as follow
$$(f.\mu)_{t}=\int_{0}^{t}\int_{E}f(s,e)\mu (ds,de)\,\,\,\,\, , (f.\mu)_{t}=\int_{0}^{t}\int_{E}f(s,e)\nu(ds,de).$$
The random measure $\tilde{\mu}$ is defined as the compensated measure of $\mu$ such that
$$ \tilde{\mu}(w,dt,de)=\mu(w,dt,de)-\nu(dt,de).$$
In particular, the stochastic integral $U.\tilde{\mu}=\int_{E} U_{s}(e)\tilde{\mu}(ds,de)$ is a local square integrable martingale, for any predictable locally integrable process $U$.
\\
We will assume the following weak representation property, for any local martingale $M$
\begin{equation}
    M=M_{0}+ \int_{0}^{.} Z_{s}.dW_{s}+\int_{0}^{.}\int_{E} U_{s}(e)\tilde{\mu}(de,ds).
\end{equation}
Now we introduce the following spaces of processes which will be often used in the sequel.\\
For any $p\geq  1$,
$\mathcal{P}$ stands for the $\sigma$-field of all predictable sets of $[0,T]\times\Omega$.\\
$\bullet$ ${\cal G}_{loc}(\mu)$ the set of $\cal P \otimes\mathcal{E}$-measurable $\R$-valued functions $H$  such that $$|H|^2.\nu_t <+\infty .$$\\
$\bullet$ $\mathbb{L}_{T}^{exp}$ the set of all $\mathcal{F}_{t}$-measurable random variables $Y$  such that $\forall \gamma>0$  $$\mathbb{E}[\exp(\gamma |Y|)]<+\infty .$$\\
$\bullet$ $\mathbb{H}^{2}([0,T])$ the set of all $\mathbb{R}$-valued c\`{a}dl\'{a}g and $\mathcal{F}_{t}$-progressively measurable  processes $Z$ such that
$$ \mathbb{E}[\int_{0}^{T}|Z_{s}|^{2}ds]<+\infty .$$\\
$\bullet$ $\mathcal{S}^{2}([0,T])$ is the space of $\mathbb{R}$ valued c\`{a}dl\'{a}g and $\mathcal{F}_{t}$-progressively measurable processes $Y$  such that $$\mathbb{E}\left[ \displaystyle{ \sup_{0\leq t\leq T}} |Y_{t}|^{2}\right]<+\infty .$$\\
$\bullet$ $\mathbb{H}^{2}_{\nu}([0,T])$ the set of all predictable processes $U$ such that $$\mathbb{E}\left[ \left(\int_{0}^{T}\int_{\mathbb{E}}|U_{s}(e)|^{2}\nu(de)dt\right) \right]<+\infty.$$
$\bullet$ $\L^{0}(B(E),\nu)$ is the set of all $B(E)$-measurable functions with the topology of convergence in measure with $ |u-u^{'}|_{t}= (\int_{E}|u(e)-u^{'}(e)|^{2}\nu (de))^{\frac{1}{2}}$.\\
\subsection{Quadratic Exponential BSDEs with jumps}
We are given the following objects: \\
$\bullet$ The terminal condition $\xi$ is an $\mathcal{F}_{T}$-measurable random variable.\\
$\bullet$ $W=(W_{t})_{t\leq T}$ be a $d$-dimensional Brownian motion.\\ 
$\bullet$ $\mu$ a random measure with compensator $\nu$ and $\tilde{\mu}(ds,de)=\mu(ds,de)-\nu(ds,de)$.\\
$\bullet$ The generator $f: \Omega\times[0,T]\times\mathbb{R}\times\mathbb{R}^{d}\times\mathbb{L}^{0}(\mathcal{B}(E),\lambda)\rightarrow \bar{\mathbb{R}}$ are always taken $\mathcal{P}\otimes\mathcal{B}(\mathbb{R}^{d+1})\otimes\mathcal{B}(\mathbb{L}^{0}(\mathcal{B}(E),\lambda))$-measurable.\\
 We consider a class of coefficient as follows \\
\begin{equation}
f_{t}(y,z,u)=\hat{f}_{t}(y,z)+\int_{E}g_{t}(u(e))\nu( de).
\end{equation}
 $\hat{f}:\Omega\times[0,T]\times \mathbb{R}^{1+d}\rightarrow\bar{\mathbb{R}}$ is a $\mathcal{P}\otimes\mathcal{B}(\mathbb{R}^{d+1})$-measurable function and $g:\Omega\times[0,T]\times\mathbb{L}^{0}(\mathcal{B}(E),\lambda)\times E\rightarrow\bar{\mathbb{R}}$ to be $\mathcal{P}\otimes\mathcal{B}(\mathbb{L}^{0}(\mathcal{B}(E),\lambda))\times B(E)$-measurable function.
\\ 
\begin{Remark}
  This family of generators was introduced by Becherer in \cite{B06} to prove the existence and uniqueness of solution of Lipschitz BSDEJ when the terminal condition is bounded.
\end{Remark}
We shall consider the following BSDEJ with data $(f,\xi)$
\begin{equation}\label{Q_{exp}}
 Y_{t}=\xi+\int_{t}^{T} f(s,Y_{s},Z_{s},U_{s})ds-\int_{t}^{T}Z_{s}dW_{s}-\int_{t}^{T}\int_{E}U(t,e)\tilde{\mu}(dt,de)\hspace{1cm}\forall t\in[0,T], \P-a.s.
\end{equation}

Let us now define a solution of Backward Stochastic differential equation with jumps.
\begin{Definition}
Let $\xi$ be a $\mathcal{F}_{T}$-measurable random variable. A solution of BSDE with jumps associated to $(f,\xi)$ is a triple $(Y,Z,U)$ of progressively measurable processes in the space $S^{2}([0,T])\times\H^{2}([0,T])\times\H^{2}_{\nu}([0,T])$ that satisfy \eqref{Q_{exp}}.

\end{Definition}
In order to get an existence result, we now state our main assumption.

\begin{Assumption}
\[
\label{H1}\left\{ \begin{array}{l}
\bullet \textrm{ Continuity condition: }\forall t\in[0,T],\mathbb{P}\hbox{-a.s},\, (y,z,u)\rightarrow f_{t}(y,z,u) \hbox{is continuous }.\\
\\
\bullet \textrm{ Integrability condition : }
\forall \gamma >0 \,\,\,\E\left[\exp{(\gamma\left(e^{C_{t,T}}|\xi|+\int_t^T e^{C_{t,s}}d\Lambda_s)\right)}\right] <+\infty.\\
\\
\bullet \textrm{ Structure condition : }
\forall(y,z,u)\in \R\times \R^{d+1}, \forall t\in [0,T]\\
\\
\hspace{1cm}             -\frac{1}{\delta}j_t(-\delta u)-{\delta\over 2}|z|^2-l_t-c_t|y|\le f(t,y,z,u)\le {1\over \delta}j_t(\delta u)+{\delta\over 2}|z|^2+l_{t}+c_t|y|.\\
\textrm{ where } l_{t} \textrm{ and } c_{t} \textrm{ are two positive continuous increasing processes.}
\\
\\
\bullet \,\,A_{\gamma}\textrm{-condition : }
\textrm{there exists a } \mathcal{P}\otimes\mathcal{B}(\mathbb{R}^{d+3})\otimes\mathcal{B}(E) \textrm{measurable function } \gamma  \textrm{ with }\\  \gamma*\widetilde \mu \in {\cal U}_{\exp} \textrm{ and } \gamma>-1 \textrm{ such that }\forall y,z,u,\bar{u}\in\R\times \R^{d+2}, \forall t\in [0,T], \P-a.s\\
\\
       \hspace{2cm}  f(t,y,z,u)-f(t,y,z,\bar u)\le \int_E \gamma_t[u(x)-\bar u(x)]\xi(t,x)\lambda(dx),\\
       
\\
\end{array} \right.
\]
\end{Assumption}
where 
\begin{equation*}
j_{t}(\delta u)=\int_{\E}\left(e^{\delta u_{s}(e)}-\delta u_{s}(e)-1\right)\nu(de).
\end{equation*}
We shall also make the following assumption needed throughout the proof of existence.
\begin{Assumption}
\[
\label{H2}\left\{ \begin{array}{l}
\bullet \,\,\forall t\in[0,T],\mathbb{P}\hbox{-a.s},\, (y,z,u)\rightarrow f_{t}(y,z,u) \hbox{is continuous }.\\
\\

\bullet \textrm{ The terminal condition } \xi \textrm{ is bounded in } \mathbb{L}^{\infty }(\mathcal{F}_{T}) \textrm{ and } f_{t}(0,0,0) \textrm{ is bounded.}\\
\\
\bullet  \textrm{ f satisfies the Lipschitz condition } \forall y,\bar{y},z,\bar{z}\in\R^{2}\times\R^{2d},  \exists C>0  \textrm{ such that } \forall t\in[0,T]\\
 \\
 \hspace{2cm }|f_{t}(y,z,u)-f_{t}(\bar{y},\bar{z},\bar{u})\leq C\left( |y-\bar{y}|+|z-\bar{z}|\right) .\\
 \\
 \bullet \textrm{ f is locally Lipschitz continuous in } u.
 \\
 \\
 \bullet\,\, A_{\gamma}\textrm{-condition : }
\textrm{there exists a } \mathcal{P}\otimes\mathcal{B}(\mathbb{R}^{d+3})\otimes\mathcal{B}(E) \textrm{-measurable function } \gamma^{y,z,u,\bar{u}} \textrm{ such that }
\\
\\
       \hspace{1.2cm}  f(t,y,z,u)-f(t,y,z,\bar u)\le \int_E \gamma_t^{y,z,u,\bar{u}}(e)[u(e)-\bar u(e)]\xi(t,e)\lambda(de),\hspace{0.5cm}\forall t\in [0,T], \P-a.s.\\
       
\\
\end{array} \right.
\]
\end{Assumption}

\begin{Remark}
The above assumptions are essential in our framework to get the existence of the solution of the BSDEJ. Here we deal with a terminal condition having a finite exponential moment of order $\gamma$ and a jump measure with infinite activity. It has been shown in \cite{ELKMN16}, under assumption $(H1)$, the existence of solution
of BSDEs with finite activity jumps. \\
When the jumps have infinite activity, Becherer \cite{B14} proved the existence and the uniqueness of the solution of bounded JBSDE when the generator $f$ satisfies the assumption $\eqref{H2}$.\\
\end{Remark}
%


\subsection{Existence of solution with infinite activity}
The question is how to prove existence of solution of BSDEJ with infinite activity jumps and unbounded terminal condition without using the seminal paper of Kobylanski.\\
 The main tool is to consider an approximation sequence of globally Lipschitz  BSDEs with finite random measure for which the existence and uniqueness of the solution are well known. To insure that the approximate coefficient conserves all the properties of $f$, we use a double approximation which means that we  consider a sequence of function bounded from below and above by linear quadratic function. We split the coefficient of the BSDE into the sum of two positive and negative functions and approximate respectively each function by inf-convolution and sup-convolution. Furthermore, the BSDEJ have infinite activity jumps,  we have to deal with some specific difficulties due to the infinite number of small jumps. The idea is then to introduce a truncated measure with $\lambda (A)<\infty$  in the auxiliary BSDE for which existence of the solution is guaranteed.
Adopting the forward approach we prove that this sequence of BSDEJ is in fact an Exponential Quadratic Semimartingales. And then by the stability theorem (\ref{Stability}) we show that the limit of those solution solves the original BSDEJ.
\\
As explained above, our point of view is based on the forward approach which is essentially standing on semimartingale. Hence, let us recall from \cite{ELKMN16} the definition of exponential quadratic semimartingale as well as the class of ${\cal S}_{Q}(|\xi|,\Lambda,C)$-semimartingale needed in the sequel.
\begin{Definition}
A Quadratic Exponential Special Semimartingale $Y$ is a c\`{a}dl\'{a}g process such that $Y=Y_{0}-V+M$ with $V$ a local finite variation process and $M:=M^{c}+M^{d}$ a local martingale part with the following structure condition $\mathcal{Q}(\Lambda, C,\delta)$:\\ There exist an increasing predictable processes $C$, $\Lambda$ and a positive constant $\delta$ such that 
\begin{equation}
-\frac{\delta}{2}d \langle  M^{c}\rangle_{t}-d\Lambda_{t}-|Y_{t}|dC_{t}-j_{t}(-\delta M_{t}^{d}])\ll  dV_{t} \ll\frac{\delta}{2}d\langle M^{c}\rangle_{t}+d\Lambda_{t}+|Y_{t}|dC_{t}+j_{t}(\Delta M_{t}^{d})_{\delta}.
\end{equation}
Note that the symbol "$\ll$" means that the difference is an increasing process.
\end{Definition}
We introduce now the  class of ${\cal S}_{Q}(|\xi|,\Lambda,C)$-semimartingale  which will play an important role in the proof of the existence result. 
 \begin{Definition}\label{Class}
 ${\cal S}_{Q}(|\xi|,\Lambda,C)$  is the class of all ${\cal Q}(\Lambda,C)$-semimartingales  $Y$ such that $$|Y_t|\le \bar\rho_t\left[e^{C_{t,T}}|Y_T|+\int_t^T e^{C_{t,s}}d\Lambda_s\right], \quad \forall t\in[0,T], \P\hbox{-a.s}.$$
 where $\bar{\rho}_{\sigma}(X_{\tau})=\ln\E\left[ \exp (X_{\tau})|\mathcal{F}_{\sigma}\right]. $
\end{Definition}
We are now in position to give the main result of this paper.
\begin{Theorem}
  Under assumption \eqref{H1}, there exists a solution $(Y,Z,U)\in{\cal S}_{Q}(|\xi|,\Lambda,C)\times\mathbb{H}^{2}\times \mathbb{H}^{2}_\nu$ of the BSDEJ \eqref{Q_{exp}}.
\end{Theorem}
\begin{proof}
 For the sake of clarity we split the proof into three main steps.\\
$\bullet$ The first step consists to introduce an auxiliary generator $f^{n,m,\kappa}$ uniformly Lipschitz $(y,z)$ and locally Lipschitz in $u$  as follow
$$f^{n,m,\kappa}(y,z,u):=\bar{f}^{n,\kappa}(y,z,u)-\underline{f}^{m,\kappa}(y,z,u).$$
From this and using a well known results, we justify the existence and the uniqueness of solution  of JBSDE associated to $(f^{n,m,\kappa},|\xi|)$. The solution will be a triple of progressively measurable processes  denoted by $(Y^{n,m,\kappa},Z^{n,m,\kappa},U^{n,m,\kappa})$.\\
$\bullet$ The second step: Since we have the same assumptions on the generator of the JBSDE and the quadratic exponential semimartingales, we prove that  the solution $Y^{n,m,\kappa}$ of the BSDEJ associated to $(f^{n,m,\kappa},|\xi|)$ is a $\mathcal{Q}(\Lambda,C,\delta)$-semimartingale.\\
$\bullet$ The last step will be the convergence of the approximated sequence of BSDE ($f^{n,m,\kappa},|\xi|$). Using the stability theorem (\ref{Stability}) which is based on the stability theorem of Barlow-Protter \cite{BP90} we prove that the limit of $(Y^{n,m,\kappa},Z^{n,m,\kappa},U^{n,m,\kappa})$ exists and solves the original BSDE.\\
\\
\textbf{$\underline{Step 1}$: Construction of the truncated sequence of BSDEJs}\\  
 For $\kappa>1$ we consider a random measure $\nu^{\kappa}$ as follows 
$$ \nu^{\kappa}(w,dt,de):=1_{ \left\lbrace |e|\geq \frac{1}{k}\right\rbrace\ }\nu(dt,de) .$$
Notice that the truncated random measure $\nu^{\kappa}$ introduced above is a finite random measure i.e for all borelian set $A$, $\nu^{\kappa}(A)<+\infty$. \\ 
 Before proceeding with the proof, we will need the following proposition which provides essential properties of $(f^{n,m,\kappa})$ needed in the proof. \\ First, Let us introduce the regularization function $b_{n}$ also known  as the approximation of Moreau and Yosida
\begin{equation}\label{reg f}
b_{n}(r,v,w)=\displaystyle{\inf_{r,v,w}} \left\lbrace n|r|+n|v|+n|w| \right\rbrace .
\end{equation}
\begin{Lemma}\label{regularisation}
Under hypothesis \eqref{H1}, let us consider the generator  $\bar{f}=f^{+}$ and $\underline{f}=f_{-}$. We define the sequence $\bar{f}^{n,\kappa}$,$\bar{q}^{n,\kappa}$, $\underline{f}^{m,\kappa}$, and $\underline{q}^{m,\kappa}$ respectively as the in-convolution and sup-convolution of $\bar{f}$, $\bar{q}$, $\underline{f}$ and $\underline{q}$ with the regularization function $b_{n}$. The regularized functions are defined as follows 
\begin{itemize}[label=\textbullet,font=\large]
\item
$\bar{f}^{n,\kappa}(y,z,u):=\bar{F}^{n}(y,z)+\bar{G}^{n,\kappa}(u):=\displaystyle{\inf_{(r,w)\in\Q^{1+d}}}\left\lbrace \hat{f}_{t}(r,w)+n|r-y|+n|w-z|\right\rbrace$
$$+\displaystyle{\inf_{v\in\Q}}\left\lbrace\int_{E} g_{t}(v(e))\zeta(t,e)\lambda^{\kappa}(de)+n|u-v|_{\nu}\right\rbrace .$$
\item
$\underline{f}^{m,\kappa}(r,v,w):=\underline{F}^{m}(y,z)+\underline{G}^{m,\kappa}(u):=\displaystyle{\sup_{(r,w)\in\Q^{d+1}}}\left\lbrace \hat{f}_{t}(r,w)+m|r-y|+m|w-z|\right\rbrace $
$$+\displaystyle{\sup_{v\in\Q}}\left\lbrace \int_{E}g_{t}(v(e))\zeta(t,e)\lambda^{\kappa}(de)+m|u-v|_{\nu}\right\rbrace .$$
\item
$\bar{q}^{n,\kappa }(t,y,z,u)=\bar{q}^{\kappa}\wedge b_{n}(y,z,u) \hbox{\,\,\,\,\,\,and\,\,\,\,\,\,} \underline{q}^{m,\kappa}(t,y,z,u)=\bar{q}^{\kappa}\vee b_{m}(y,z,u).$\
\end{itemize}
We have the following essential properties
\begin{enumerate}
\item $(\bar{f}^{n,\kappa})$, $(\bar{q}^{n,\kappa})$,$(\underline{f}^{n,\kappa})$, $(\underline{q}^{n,\kappa})$ are increasing (resp. decreasing) sequences in $n,\kappa$ (resp. $m$).
\item The sequences $(\bar{f}^{n,\kappa})$, $(\bar{q}^{n,\kappa})$,$(\underline{f}^{m,\kappa})$, $(\underline{q}^{m,\kappa})$ are globally Lipschitz continuous in $(y,z)$ and locally Lipschitz in $u$ for each $n,m,\kappa .$
\item The sequences $(\bar{f}^{n,\kappa})$, $(\bar{q}^{n,\kappa})$ converge resp. to $\bar{f}$ and $\bar{q}$ as $n,\kappa$ goes to $\infty$ i.e
$$\bar{f}^{n,\kappa }\nearrow \bar{f}\,\,\,\,\hbox{and} \,\,\,\, \bar{q}\nearrow \bar{q}.$$
The convergence is also uniform.
\item The sequences $(\underline{f}^{m,\kappa})$, $(\underline{q}^{m,\kappa})$ converge respectively to $\underline{f}$ and $\underline{q}$ as $m,\kappa$ goes to $\infty$ i.e
$$\underline{f}^{m,\kappa }\nearrow \underline{f}\,\,\,\,\hbox{and} \,\,\,\, \underline{q}^{m,\kappa}\nearrow \underline{q}.$$
 \item  $(f^{n,m,\kappa})_{n,m,\kappa}$ satisfy the structure condition of Assumption \eqref{H1}.
\end{enumerate}


\end{Lemma}

\begin{Remark}
We emphasize that the regularization technique developed in proposition \eqref{regularisation} were inspired by the ones developed in the paper \cite{ELKMN16}. Nonetheless, as explained throughout the paper, the main difficulty in carrying out this construction is that the structure of this equations \eqref{Q_{exp}} are defined in infinite activity setting.
\end{Remark}
Let as now introduce  the BSDE associated to the truncated measure $\nu^{\kappa}$
\begin{equation}\label{measure trun}
  dY_{t}^{\kappa}=f^{\kappa}_{t}(Y_{t}^{\kappa},Z^{\kappa}_{t},U_{t}^{\kappa})dt-Z^{\kappa}_{t}dW_{t}+\int_{E}U^{\kappa}_{t}(e)\tilde{\mu}(dt,de)\,\,\,\,\,\, t\in[0,T],\P-a.s.
\end{equation}
where $f^{\kappa}_{t}(y,z,u)=\hat{f}_{t}(y,z)+\int_{E}g_{t}(u(e))\nu^{\kappa}(de)$.\\
The following BSDE is driven jointly by the Brownian motion $W$ and the truncated measure $\mu^{\kappa}$. The associated filtration is denoted by $\mathcal{F}_{t}^{\kappa}\subset \mathcal{F}_{t}$.\\
Note that (\ref{measure trun}) is an exponential quadratic BSDE with finite activity jumps, the generator $f^{\kappa}$ satisfies the same hypotheses as $f$.\\
\\
After we construct the BSDE associated to the truncated random measure, we introduce the intermediate BSDE $(f^{n,m,\kappa },|\xi|)$. $\forall t\in[0,T]$, $\P$-a.s
 \begin{equation}\label{Truncated BSDE}
-dY_{t}^{n,m,\kappa }=f_{t}^{n,m,\kappa}(Y^{n,m,\kappa}_{t},Z^{n,m,\kappa}_{t},U^{n,m,\kappa}_{t})dt-Z_{t}^{n,m,\kappa }dW_{t}-\int_{E}U^{n,m,\kappa}(e)\tilde{\mu}(dt,de).
 \end{equation}
First of all we have to justify the existence of solution to this BSDE. \\In fact this is a simple consequence of the existence results of \cite{B14}. Thanks to the above lemma our coefficient $f^{n,m,\kappa}$ satisfies the assumption \eqref{H2}. It remains to show that the $A_\gamma$-condition hold for $(f^{n,m,\kappa})$.  Let  $u,\bar{u}\in\mathbb{L}^{0}(\mathcal{B}(E),\nu)$ and $ y,z\in\mathbb{R}, \mathbb{R}^{d}$ such that
\begin{align*}
f^{n,m,\kappa }(t,y,z,u)-f^{n,m,\kappa }(t,y,z,\bar u)&:=[G^{n,\kappa}(t,u)-G^{n,\kappa}(t,\bar{u})]+[G^{m,\kappa}(t,u)-G^{m,\kappa}(t,\bar{u})]\\
&\leq \bar{q}^{n,\kappa}(y,z,u)-\bar{q}^{n,\kappa}(y,z,\bar{u})+\underline{q}^{m,\kappa}(y,z,u)-\underline{q}^{m,\kappa}(y,z,\bar{u}).
\end{align*}
Following \cite{ELKMN16}, we know that  $\bar{q}^{n,\kappa}$ and $\underline{q}^{m,\kappa}$ satisfy respectively the $A_{\gamma}$-condition. Hence 
\begin{align*}
\bar{q}^{n,\kappa}(y,z,u)-\bar{q}^{n,\kappa}(y,z,\bar{u})+\underline{q}^{m,\kappa}(y,z,u)-\underline{q}^{m,\kappa}(y,z,\bar{u})&\leq \int_{E}\gamma^{n}[u(e)-\bar{u}(e)]\nu^{\kappa}(de)\\
&+\int_{E}\gamma^{m}[u(e)-\bar{u}(e)]\nu^{\kappa}(de).\\
\end{align*}
where $-1 <\gamma^{n}< n $ and $-1< \gamma^{m}< m$. Therefore 
$$f^{n,m,\kappa }(t,y,z,u)-f^{n,m,\kappa }(t,y,z,\bar u)\leq \int_{E}\gamma^{n,m}(u(e),\bar{u}(e))(u(e)-\bar{u}(e)\nu^{\kappa}(de).$$
According to \cite{B14}, which deals with the lipschitz BSDE with jumps, we know that the BSDE \eqref{Truncated BSDE}  has a unique solution $(Y^{n,m,\kappa},Z^{n,m,\kappa},U^{n,m,\kappa})$.
\\
Moreover since $A_{\gamma}$-condition holds for $(f^{n,m,\kappa})_{n,m,\kappa}$, we can also apply the comparison theorem, to obtain that
 for all $t\in[0,T]$, $\P$-a.s
\begin{equation}\label{Comparison sln}
  Y_{t}^{n,m+1,\kappa}\leq Y_{t}^{n,m,\kappa }\leq Y_{t}^{n,m+1,\kappa+1 }.
\end{equation}
\textbf{\underline{Step 2: Construction of the sequence of $\mathcal{Q}(C,\Lambda,\delta )$-semimartingale }}\\
\\
In view of proposition \eqref{regularisation} and assumption \eqref{H1} we have  $$ \underline{q}_{t}\leq \underline{q}_{t}^{\kappa,m}\leq f_{t}^{n,m,\kappa}\leq \bar{q}_{t}^{\kappa,n}\leq \bar{q}_{t}\quad \forall t\in[0,T],\quad \P\hbox{-a.s.}$$
Thus, we know that all the requirement of the definition \eqref{Class} are full filled. It follow that 
$(Y^{n,m,\kappa})_{n,m,\kappa}$ defined as the unique solution of the BSDEJ \eqref{Truncated BSDE} is a $\mathcal{Q}(C,\Lambda,\delta )$-semimartingale.

Moreover from Lemma $(3.3)$ we know that for all stopping times $\sigma\leq T$ 
  $$Y_{\sigma}^{n,m,\kappa}\leq\bar{\rho}_{\sigma}\left[ e^{C_{\sigma, T}}|\xi|+\int_{\sigma}^{T}e^{C_{\sigma,s}}d\Lambda_{s}\right] \hspace{2cm}\P\hbox{-a.s}.$$

\textbf{\underline{Step 3: The convergence of the semimartingale }}\\
\\
The idea is now to prove that the limit in some sens of those sequence of $\mathcal{Q}(C,\Lambda,\delta )$-semimartingale is a solution of the BSDEJ \eqref{Q_{exp}}.\\
$\bullet$ We know from the first step that the sequence $(Y^{n,m,\kappa})_{\kappa}$ is bounded and increasing  in $\kappa$ then $(Y^{n,m,\kappa})_{\kappa}$ converge in $\H^{2p}(\R)$ to $Y^{m,n }$ such that $$Y^{n,m,\kappa}\nearrow Y^{n,m }\,\,\,\,\,\,\,\, \kappa \,\,\,\,\hbox{goes to}\,\,\,\, \infty .$$
Thanks to the Dini's lemma \cite{Del79} the convergence is uniform. Hence it follows from Proposition (\ref{Stability})  that $Y^{n,m}$ is a $\mathcal{Q}(\Lambda,C,\delta)$-semimartingale and satisfies
\begin{equation}
\label{estimate:variation1}
\mathbb{E} \Big[\int_0^T |dV^{n,m}_s| \, \Big] \leq C, \quad \mbox{and} \quad \mathbb{E} \, \big[ \big(M^{n,m} \big)^{\star} \big] \leq C.\nonumber
\end{equation}

\begin{equation}
\label{limite:variation1}
\lim_{ \kappa \to \infty} \mathbb{E} \, \big[ \big(V^{n,m,\kappa} - V^{n,m}\big)^* \big] = 0 \quad \mbox{and} \quad  \lim_{\kappa\to\infty} \|M^{n,m,\kappa} - M^{n,m}\|_{\mathcal{H}^1} = 0.\nonumber
\end{equation}
$\bullet$ We proceed exactly as \cite{ELKMN16} since $(Y^{n,m})_{n}$ and $(Y^{n,m})_{m}$ are monotone bounded uniformly sequences. Therefore they converge monotonically to some process $Y$ i.e   $$ \displaystyle{\lim_{n,m}}\searrow Y^{n,m}= Y.$$
 Using the same arguments  $Y$ is a  $\mathcal{Q}(\Lambda,C,\delta)$-semimartingale. and the following estimates holds
\begin{equation}
\label{estimate:variation2}
\mathbb{E} \Big[\int_0^T |dV_s| \, \Big] \leq C, \quad \mbox{and} \quad \mathbb{E} \, \big[ \big(M \big)^{\star} \big] \leq C.\nonumber
\end{equation}

\begin{equation}
\label{limite:variation2}
\lim_{ n,m \to \infty} \mathbb{E} \, \big[ \big(V^{n,m}-V\big)^* \big] = 0 \quad \mbox{and} \quad  \lim_{ n,m \to \infty} \|M^{n,m } - M\|_{\mathbb{H}^1} = 0.\nonumber
\end{equation}
$\bullet$ It remains to show that the $\mathcal{Q}(\Lambda,C,\delta)$-semimartingale $Y$ is the solution of the \emph{exponential quadratic BSDEJ} \eqref{Q_{exp}} such that \\
\begin{equation}
f(Y,Z,U)=\displaystyle{\lim_{n,m,\kappa}}f^{n,m,\kappa}(Y^{n,m,\kappa},Z^{n,m,\kappa},U^{n,m,\kappa})\,\,\,\,\,\,\,\,\,\, d\mathbb{P}\otimes d\nu \,\,\,\,\hbox{a-s}\nonumber
\end{equation}
\begin{equation}
(Z^{n,m,\kappa}.W+U^{n,m,\kappa}\star\tilde{\mu})=\displaystyle{\lim_{n,m,\kappa}}(Z.W+U\star\tilde{\mu}).\nonumber
\end{equation}
We need to define a sequence of stopping times $\tau_{l}$ related to the class  $\mathcal{Q}(\Lambda,C,\delta)$ such that $\tau_{l}$ goes to $\infty$ for a large $l$. Let us fix $L\in\N^{*}$ such that 
$$\tau_{l}:=\displaystyle{\inf_{t\geq 0}}\left\lbrace \E\left[ \exp(e^{C_{T}}|\xi|+\int_{0}^{T}e^{C_{s}}d\Lambda_{s}|\mathcal{F}_{t}\right] >l\right\rbrace .$$
According to the first part of the proof, the monotone convergence of $(Y^{n,m,\kappa})_{.\wedge \tau_{l}}:=((M^{n,m,\kappa})^{c}+V^{n,m,\kappa}+U^{n,m,\kappa}.\tilde{\mu})_{.\wedge\tau_{l}}$ is uniform. Therefore, we can extract a subsequence of $(M^{n,m,\kappa})_{.\wedge\tau_{l}}$ which converges strongly to $M_{.\wedge\tau_{l}}$
$$(M^{n,m,\kappa})_{t\wedge\tau_{l}}=Z^{n,m,\kappa}_{t}1_{\left\lbrace t\leq \tau_{l}\right\rbrace}.W+U^{n,m,\kappa}1_{\left\lbrace t\leq \tau_{l}\right\rbrace}.\tilde{\mu}.$$
Once again, we can subtract a subsequence $Z^{n,m,\kappa}_{t}1_{\left\lbrace t\leq \tau_{l}\right\rbrace}$ and $U^{n,m,\kappa}1_{\left\lbrace t\leq \tau_{l}\right\rbrace}$ converge almost surely to $Z$ and $U$ in $\H^{2}\times\H^{2}_{\nu}$
 
$$dV^{n,m,\kappa}_{.\wedge\tau_{l}}:=\hat{f}^{n,m,\kappa}_{t}(Y_{t\wedge\tau_{l}}^{n,m,\kappa},Z_{t\wedge\tau_{l}}^{n,m,\kappa})1_{\left\lbrace t\leq\tau_{l}\right\rbrace}dt+G^{n,m,\kappa}(U_{t\wedge\tau_{l}}^{n,m,\kappa})1_{\left\lbrace t\leq\tau_{l}\right\rbrace}.$$\\
where $\hat{f}^{n,m,\kappa}(y,z)=\bar{F}^{n}(y,z)+\underline{F}^{m}(y,z)$ and $G^{n,m,\kappa}(u)=\bar{G}^{n,\kappa}(u)+\underline{G}^{m,\kappa}(u)$. 
\\
\\
$\bullet$ As a last step, we will show that $f^{n,m,\kappa}(Y^{n,m,\kappa},Z^{n,m,\kappa},U^{n,m,\kappa})$ converge to $f(Y,Z,U)$ in $\mathbb{L}^{1}( d\P\otimes d\nu\otimes dt)$.
In fact as $Z^{n,m,\kappa}$ and $U^{n,m,\kappa}$ are unbounded,  one can decompose the expression  above in 2 quantities: one in the region where $\left\lbrace|Z^{n,m,\kappa}|+|U^{n,m,\kappa} | \leq C\right\rbrace$ and the other in the region $\left\lbrace|Z^{n,m,\kappa}|+|U^{n,m,\kappa}| > C \right\rbrace$.
\begin{align*}
&\displaystyle{\E\left[\int_{0}^{\tau_{l}}|f^{n,m,\kappa}(Y_{s}^{n,m,\kappa},Z_{s}^{n,m,\kappa},U_{s}^{n,m,\kappa})-
f_{s}(Y_{s},Z_{s},U_{s})|ds\right]}\\
=&\displaystyle{ \E\left[\int_{0}^{\tau_{l}}|f^{n,m,\kappa}(Y_{s}^{n,m,\kappa},Z_{s}^{n,m,\kappa},U_{s}^{n,m,\kappa})-
f_{s}(Y_{s},Z_{s},U_{s})|1_{\left\lbrace|Z^{n,m,\kappa}|+|U^{n,m,\kappa} | \leq C\right\rbrace}ds\right]=:A_{1}}\\
+&\displaystyle{\E\left[\int_{0}^{\tau_{l}}|f^{n,m,\kappa}(Y_{s}^{n,m,\kappa},Z_{s}^{n,m,\kappa},U_{s}^{n,m,\kappa})-
f_{s}(Y_{s},Z_{s},U_{s})|1_{\left\lbrace|Z^{n,m,\kappa}|+|U^{n,m,\kappa} | > C\right\rbrace}ds\right]=:A_{2}}\\
\end{align*}
We start by studying the first term $A_{1}$. Observe that in the region $\left\lbrace|Z^{n,m,\kappa}|+|U^{n,m,\kappa} | \leq C\right\rbrace$, since $Y^{n,m,\kappa}$ is bounded in $[0,\tau_{l}]$ 
\begin{align*}
A_{1}&=\displaystyle{\E\left[\int_{0}^{\tau_{l}}|f^{n,m,\kappa}(Y_{s}^{n,m,\kappa},Z_{s}^{n,m,\kappa},U_{s}^{n,m,\kappa})-
f_{s}(Y_{s},Z_{s},U_{s})|1_{\left\lbrace|Z^{n,m,\kappa}|+|U^{n,m,\kappa} | \leq C\right\rbrace}ds\right]}\\
&\displaystyle{\leq \phi_{s}+\E\left[\int_{0}^{\tau_{l}}|G^{n,\kappa}(u)-G(u)|1_{\left\lbrace|Z^{n,m,\kappa}|+|U^{n,m,\kappa} | \leq C\right\rbrace}ds\right]}\\
&\displaystyle{\leq \phi_{s}+E\left[ \int_{0}^{\tau_{l}}j^{\kappa}_{s}(\delta U^{n,m,\kappa}(e))-j_{s}(\delta U(e))1_{\left\lbrace |U^{n,m,\kappa}|\leq C\right\rbrace }ds \right]}.
 \end{align*}
 Therefore, to prove that $ f^{n,m,\kappa}(Y^{n,m,\kappa},Z^{n,m,\kappa},U^{n,m,\kappa})-f(Y,Z,U)$ is bounded in $\L^{1}$, it is sufficient to show that 
$\E\left[ \int_{t}^{T}j^{\kappa}_{s}(\delta U^{n,m,\kappa}(e))-j_{s}(\delta U(e))1_{\left\lbrace |U^{n,m,\kappa}|\leq C\right\rbrace }ds \right]$ converge to zero 
as $n,m,\kappa\rightarrow+\infty$. and by Dominated convergence theorem we get the desire result.\\ We start by following a technique similar to the one used in the proof of Theorem $(4.13)$  of \cite{B14}. Observe that $j(\delta U(e))=j(\delta U(e))1_{\left\lbrace |e|\leq \kappa \right\rbrace }+j(\delta U(e))1_{\left\lbrace |e|\geq \kappa \right\rbrace }$.\\
Hence one can write 
\begin{align*}
&\E\left[ \int_{t}^{\tau_{l}}j^{\kappa}_{s}(\delta U^{n,m,\kappa}(e))-j_{s}(\delta U(e))1_{\left\lbrace |U^{n,m,\kappa}|\leq C\right\rbrace }ds \right]\\
&\leq  \E\left[ \int_{t}^{\tau_{l}}j^{\kappa}_{s}(\delta U^{n,m,\kappa}(e))-j^{\kappa}_{s}(\delta U(e))1_{\left\lbrace |U^{n,m,\kappa}|\leq C\right\rbrace }1_{\left\lbrace |e|\geq \frac{1}{\kappa}\right\rbrace}\nu(de,ds)\right]\nonumber\\
&+\E\left[ \int_{t}^{\tau_{l}}j_{s}(\delta U(e))1_{\left\lbrace |U^{n,m,\kappa}|\leq C\right\rbrace } 1_{\left\lbrace|e|\leq \frac{1}{\kappa}\right\rbrace}\nu(de,ds)\right].
\end{align*}
The first term in the above inequality tend to zero since $j^{\kappa}_{t}(\delta U^{n,m,\kappa})-j^{\kappa}_{t}(\delta U)$ is uniformly bounded in $\mathbb{L}^{1}$. For the last one, using the fact that $1_{\left\lbrace|e|\leq \frac{1}{\kappa}\right\rbrace}$ tend to zero as $\kappa\rightarrow 0$, we see that it also tend to zero. 
\\ 
$\bullet$ Finally let us study the term $A_{2}$. By Tchebychev inequality we have \\$$\E \left[ 1_{\left\lbrace |Z^{n,m,\kappa}|+|U^{n,m,\kappa}|\geq C\right\rbrace}\right]\leq \frac{2}{C^{2}}\E\left[ |Z^{n,m,\kappa}|^{2}+|U^{n,m,\kappa}|^{2}\right].$$
Hence, for $t\leq \tau_{l}$ from the dominated convergence theorem $f^{n,m,\kappa}(Y^{n,m,\kappa},Z^{n,m,\kappa},U^{n,m,\kappa})$ converge to $f(Y,Z,U)$ in $\mathbb{L}^{1}( d\P\otimes d\nu\otimes dt)$. Thus for $t\leq \theta_{l}$ $dV_{t}=f(t,Y,Z,U)dt$.\\
\textbf{Conclusion}\\
The exponential quadratic semimartingale $Y$ is such that $Y=Y_{0}+V+M^{c}+U.\tilde{\mu}$
is solution of the original BSDE \eqref{Q_{exp}} in $S^{2}\times\mathbb{H}^{2}\times\mathbb{H}^{2}_{\nu}$.
\end{proof}
\appendix
\section{Appendix}
\subsection{Proof of lemma \eqref{regularisation}}
$(i)$  We can start to notice that, due to the properties of the inf-convolution  of the sequence $(\bar{f^{n,\kappa}})_{n}$ and $(\bar{q}^{n,\kappa})_{n}$ are increasing. Moreover $(\underline{f}^{m,\kappa})_{m}$ and $(\underline{q}^{m,\kappa})_{m}$  are decreasing. The monotonicity of the coefficient arises from the regularization function $b$. The main point for the monotonicity of the coefficients in $\kappa$ is to notice that $g(v(e))1_{\left\lbrace|e|\geq\frac{1}{\kappa}\right\rbrace}$ is smaller then $g(v(e))$.\\
$(ii)$ To prove that $(\bar{f}^{n,\kappa})_{n,\kappa}$ is uniformly lipschitz in $(y,z)$,  we consider the function $\tilde{f}^{n,\kappa}$ such that  $\forall\epsilon >0$ , $y_{1},y_{2}$, $z_{1},z_{2}$ and $y_{\epsilon}\in\Q$, $z_{\epsilon}\in\Q^{d}$ we have
$$f^{n,\kappa}(t,y^{1},z^{1},u)\geq \tilde{f}^{n,\kappa}_{t}(y_{\epsilon},z_{\epsilon},u)+n|y-y_{\epsilon}|+|z-z_{\epsilon}|-\epsilon .$$
where $\tilde{f}^{n,\kappa}_{t}(y_{\epsilon},z_{\epsilon},u):=\inf_{v\in\Q}\left\lbrace \int_{E}g_{t}(v(e))\nu^{\kappa}(de)+n|u-v|_{\nu}+\hat{f}(r,w)\right\rbrace .$\\
\begin{align}
f^{n,\kappa}(t,y^{1},z^{1},u) &\geq \tilde{f}(y_{\epsilon}, z_{\epsilon},u) +n|y^{2}-y_{\epsilon}|+n|z^{2}-z_{\epsilon}|+n|y^{1}-y_{\epsilon}|+n|z^{1}-z_{\epsilon}|+\epsilon \nonumber\\
&\geq \tilde{f}(y_{\epsilon}, z_{\epsilon},u) -n|y^{1}-y^{2}|-n|z^{1}-z^{2}|+n|y^{2}-y_{\epsilon}|+n|z^{2}-z_{\epsilon}|+\epsilon . \nonumber
\end{align}
Hence,\\
$$f^{n,\kappa}(t,y^{1},z^{1},u)\geq f^{n,\kappa} (t,y^{2},z^{2},u)- n|y^{1}-y^{2}|-n|z^{1}-z^{2}|+\epsilon .$$
Indeed, by arbitrariness of $\epsilon$ and by interchanging respectively the roles of $(y^{1}$, $z^{1})$ and $(y^{2}$, $z^{2})$ we get the desire result. The same argument remains valid for  $(\underline{f}^{m,\kappa})$ and  $(\underline{q}^{m,\kappa})$. \\
- Let $u,\bar{u}\in\mathbb{L}^{0}(\mathcal{B}(E),\nu)$  such that
\begin{align*}
&f^{n,\kappa}(y,z,u)-f^{n,\kappa}(y,z,\bar{u}):=(\bar{F}^{n}(y,z)-G^{n,\kappa}(u))-(\bar{F}^{n}(y,z)+G^{n,\kappa}(\bar{u}))\nonumber\\
&=\displaystyle{\inf_{v\in\Q}}\left\lbrace\int_{E}g_{t}(v(e))1_{\left\lbrace|e|\geq \frac{1}{\kappa}\right\rbrace}\lambda(de)+n|u-v|_{\nu}\right\rbrace
-\displaystyle{\inf_{v\in\Q}}\left\lbrace\int_{E}g_{t}(v(e))1_{\left\lbrace|e|\geq \frac{1}{\kappa}\right\rbrace}\lambda(de)+n|\bar{u}-v|_{\nu}\right\rbrace .\nonumber
\end{align*}
Notice that
$$\displaystyle{\inf_{v\in\Q}}(f(v))-\displaystyle{\inf_{v\in\Q}}(g(v))\leq\displaystyle{\sup_{v\in\Q}}(f(v)-g(v)) .$$ \\ Thus we obtain
\begin{align}
f^{n,\kappa}(y,z,u)-f^{n,\kappa}(y,z,\bar{u})\leq\,&\displaystyle{\sup_{v\in\Q}}\left\lbrace n|u-v|_{\nu}-n|\bar{u}-v|_{\nu}\right\rbrace\leq \displaystyle{\sup_{v\in\Q}}\left\lbrace \int_{E}n|u-v|^{2}-n|\bar{u}-v|^{2}\zeta(t,e)\lambda(de)\right\rbrace\nonumber\\
\leq\,& n\int_{E}\displaystyle{\sup_{v\in\Q}}\left\lbrace|u+\bar{u}-2v|\rbrace\right|u-\bar{u}|\zeta(t,e)\lambda(de)\nonumber\\
\leq& n\int_{E}(|u|+|\bar{u}|)|u-\bar{u}|\zeta(t,e)\lambda(de).\nonumber\\
\end{align}
Then it follows from the Cauchy-Schwartz inequality that 
$$f^{n,\kappa}(y,z,u)-f^{n,\kappa}(y,z,\bar{u})\leq n \left\lbrace\|u\|_{t}+n \|\bar{u}\|_{t}\right\rbrace \|u-\bar{u}\|_{t} .$$
Hence the result.\\
$(iii), (iv)$ The convergence of the sequences in $n,m$ results  from to  the regularization function $b$. When we fix $n$ and $m$, the convergence of the sequence is only obtained by the convergence of the second part in the expression of $(\bar{f}^{n,\kappa})_{\kappa}$. \\Hence,
\begin{align}\bar{f}^{n,\kappa}(y,z,u)-\bar{f}^{n}(y,z,u)&=\displaystyle{\inf_{v\in\Q}}\left\lbrace\int_{\left\lbrace|e|\geq \frac{1}{\kappa}\right\rbrace} g_{t}(v(e))\zeta(t,e)\lambda(de)+n|u-v|_{\nu}\right\rbrace\nonumber \\
&-\displaystyle{\inf_{v\in\Q}}\left\lbrace\int_{E} g_{t}(v(e))\zeta(t,e)\lambda(de)+n|u-v|_{\nu}\right\rbrace\nonumber\\
&\leq \displaystyle{\sup_{v\in\Q}}\left\lbrace\int_{\left\lbrace|e|\geq\frac{1}{\kappa}\right\rbrace}g_{t}(v(e))\lambda(de)-\int_{E}g_{t}(v(e))\lambda(de)\right\rbrace\nonumber\\
&\leq \displaystyle{\sup_{v\in\Q}}\left\lbrace\int_{E}g_{t}(v(e))1_{[-\frac{1}{\kappa},\frac{1}{\kappa}]}\lambda(de)\right\rbrace .\nonumber
\end{align}
Since $\displaystyle{\sup_{v\in\Q}}\, g_{t}(v(e))1_{[-\frac{1}{\kappa},\frac{1}{\kappa}]}$ converge to zero, the result follows from the standard dominated convergence Theorem.\\
$(v)$ Finally 
\begin{center}
$
\begin{cases}

  &\bar{f}^{n,\kappa}\leq \bar{f}^{n}\leq \bar{q^{n,\kappa}}\leq \bar{q} \nonumber\\
     &\underline{f}^{m,\kappa}\leq\bar{f}^{m}\leq\underline{q}^{m,\kappa}\leq \underline{q}\nonumber
         \end{cases}$
$\,\,\,\,\,\Rightarrow\underline{q}\leq f^{n,m,\kappa}:=\bar{f}^{n,\kappa}-\underline{f}^{m,\kappa}\leq\bar{q}.$
\end{center}
\subsection{Exponential quadratic semimartingale and its properties}
Our goal is to prove the existence of a maximal solution of the quadratic BSDEJ using the tools introduced in \cite{ELKMN16}. Adopting this approach we summarize in this section the essential properties of the quadratic exponential semimartingales as well as a stability result   which we shall use for the construction of the solution the BSDE.
\begin{Definition}
A Quadratic Exponential Special Semimartingale $Y$ is a c\`{a}dl\'{a}g process such that $Y=Y_{0}-V+M$ with $V$ a local finite variation process and $M$ a local martingale part with the following structure condition $\mathcal{Q}(\Lambda, C,\delta)$:\\ for an increasing predictable processes $C$, $\Lambda$ and a constant $\delta$
\begin{equation}\label{structure}
-\frac{\delta}{2}d \langle  M^{c}\rangle_{t}-d\Lambda_{t}-|Y_{t}|dC_{t}-j_{t}(-\delta M_{t}^{d}])\ll  dV_{t} \ll\frac{\delta}{2}d\langle M^{c}\rangle_{t}+d\Lambda_{t}+|Y_{t}|dC_{t}+j_{t}(\Delta M_{t}^{d})_{\delta}.
\end{equation}
with 
\begin{equation*}
j_{t}(\delta u)=\int_{\E}\frac{\exp{\delta u(e)}-\delta u(e)-1}{\delta}\nu(de).
\end{equation*}
Note that the symbol "$\ll$" means that the difference is an increasing process.
\end{Definition}
We start by an example of this class of semimartingale.
\begin{Example}
The structure condition \eqref{structure} holds in the cases below:
 \begin{itemize}
\item A semimartingale $Y$ where the finite variation process $V$ is given by  $V_{t}=\frac{1}{2} \langle M^{c}\rangle_{t}+j_{t}(\delta M_{t}^{d})$ is a exponential quadratic  semimartingale.  
\item  If the finite variation part of  a semimartingale $Y$ satisfies
$$ -\frac{1}{2} \langle M^{c}\rangle_{t}-j_{t}(-\delta M_{t}^{d}) \ll V_{t}\ll \frac{1}{2} \langle M^{c}\rangle_{t}+j_{t}(\delta M_{t}^{d}).$$ 
then $Y$ is a exponential quadratic semimartingale.
\end{itemize}
\end{Example}
The following proposition gives us a characterization of a canonical class of quadratic exponential semimartingale. The canonical quadratic semimartingale is a semimartingale with $V= -{1\over 2}\langle \bar M^c\rangle_t-j(-\Delta M_{t}^{d}).\nu_t$ or $V= {1\over 2}\langle \bar M^c\rangle_t+j(\Delta M_{t}^{d}).\nu_t$ . This characterization will be useful in the sequel. 
\begin{Proposition}\cite{ELKMN16}\\ Let $\bar M=\bar M^c+\bar U.\widetilde \mu$ and $\underline{M}=\underline {M}^c+\underline{U}.\widetilde \mu$ two {c\`adl\`ag} local martingales such that  $\bar M^c+(e^{\bar U}-1).\widetilde \mu$ and 
$-\underline{M}^c+(e^{-\underline{U}}-1).\widetilde \mu$
are still {c\`adl\`ag} local martingales. Let define the canonical local quadratic exponential semimartingale:
\begin{equation*}
\begin{split}
&r(\bar M)=r(\bar M_0)+\bar M_t-{1\over 2}\langle \bar M^c\rangle_t-(e^{\bar U}-\bar U-1).\nu_t.\\
&\underline{r}(\underline{M})=\underline{r}(\underline{M}_0)+\underline{M}_t+{1\over 2}\langle \underline{M}^c\rangle_t+(e^{-\underline{U}}+\underline{U}-1).\nu_t.
\end{split}
\end{equation*}
then the following processes:
$$\exp[r(\bar M)-r(\bar M_0)]={\cal E}\left(\bar M^c+(e^{\bar U}-1).\widetilde\mu\right) \hbox{ and }\exp[-\underline{r}(\underline{M})+\underline{r}(\underline{M}_0)]={\cal E}\left(-\underline{M}^c+(e^{-\underline{U}}-1).\widetilde\mu\right)$$
\noindent are positive local martingales.
\end{Proposition}
 \begin{Proposition} Let $\psi_T \in {\cal F}_T$ such that $\exp(|\psi_T|)\in \L^1$ and consider the two dynamic risk measures:
$$\bar\rho_t(\psi_T)=\ln\left[\mathbb{E}\left(\exp(\psi_T)\vert{\cal F}_t\right)\right],\hbox{ and } \underline{\rho}_t(\psi_T)=-\ln\left[\mathbb{E}\left(\exp(-\psi_T)\vert {\cal F}_t\right)\right].$$
There exists local martingales $\bar M=\bar M^c+\bar U.\widetilde \mu$ and $\underline{M}=\underline{M}^c+\underline{U}.\widetilde \mu$ such that:
\begin{equation*}
\begin{split}
&-d\bar\rho_t(\psi_T)=-d\bar M_t+{1\over 2}d\langle \bar M^c\rangle_t+\int_E(e^{{\bar U}(s,x)}-\bar U(s,x)-1).\nu(dt,dx),\quad \bar\rho_T(\psi_T)=\psi_T.\\&
-d\underline{\rho}_t(\psi_T)=-d\underline{M}_t-{1\over 2}d\langle \underline{M}^c\rangle_t-\int_E(e^{-\underline{U}(s,x)}+\underline{U}(s,x)-1).\nu(dt,dx),\quad \underline{\rho}_T(\psi_T)=\psi_T.
\end{split}
\end{equation*}
Moreover the local martingales $\bar M^c+(e^{\bar U}-1).\widetilde\mu$ and $-\underline{M}^c+(e^{-\underline{U}}-1).\widetilde\mu$ belong to ${\cal U}_{\exp}$. The dynamic risk measures  $\bar\rho(\psi_T)$ and $\underline{\rho}(\psi_T)$ are uniformly integrable canonical quadratic exponential semimartingales.
 \end{Proposition}
 
\subsection{Integrability of the $Q(\Lambda,C)$-semimartingale}
In this part we want to investigate the integrability of this classs of semimartingale. This result will be extremely useful in the section $3$.\\
First, we consider the following decomposition of a $Q(\Lambda,C)$-semimartingale $Y$
\begin{equation}\label{transformation Y}
\bar X^{\Lambda,C}_t(|Y|):=e^{C_t}|Y_{t}|+\int_0^t e^{C_s}d\Lambda_s.
\end{equation} 
To explore to exponential integrability of this class of semimartingale , we proceed analogously to the proof of proposition  $(3.2)$ in \cite{BElK11}. Note that this decomposition appeared for the first time in the continuous setting in this paper.\\ For this propose, let us start by the definition of the $Q$-submartingale.
\begin{Definition}
A semimartingale $Y=Y_{0}-V+M$ is a $Q$-submartingale if $-V+\frac{1}{2}\langle M^{c}\rangle +\frac{1}{2}j_{t}(\Delta M^{d})$ is a predictable increasing process.
\end{Definition}
\begin{Theorem} let $\bar{X}$ be a c\`adl\`ag process given by \eqref{transformation Y} such that $\bar X^{\Lambda,C}_T(e^{|Y|})\in \mathbb{ L}^{1}$, 
the process $\bar X^{\Lambda,C}_T$ is an ${\cal Q}(\Lambda,C)$-semimartingale which belonging to  ${\cal D}_{exp}$ if and only if for any stopping times $\sigma\le \tau\le T$:
    \begin{equation}
      |Y_{\sigma}|\leq \rho_{\sigma}\left( e^{C_{\sigma,\tau}}|Y_{\tau}|+\int_{\sigma}^{\tau}e^{C_{\sigma,t}}d\Lambda_{t} \right).
    \end{equation}
\end{Theorem}
\begin{proof}
First we check that $\bar X^{\Lambda,C}_t(|Y|)$ is a $Q$-submartingale.
Applying It$\hat{o}$-Tanaka formula we get for all $t\in[0,T]$
\begin{equation}
d|Y|_{t}=sign(Y_{t_{-}})dM_{t}-sign(Y_{t_{-}})dV_{t}+(|Y_{t_{-}}+U_{t}|-|Y_{t_{-}}|)\star\mu_{t}+L_{t}^{Y}.
\end{equation}
where  $L^{Y}$ is a local time of $Y$ in zero. We denote by $M^{s}=sign(Y).dM$ and $V^{s}= sign(Y).dV$
\begin{align}
\vspace{0.5cm}
d\bar{X}^{\Lambda,C}_{t}&=e^{C_{t}}\left[ |X_{t}|dC_{t}-dV^{s}_{t}+dM^{s}_{t}+d(|Y_{t_{-}}+U_{t}|-|Y_{t_{-}}|)\star\mu+L_{t}^{Y}\right]\nonumber\\
\vspace{0.5cm}
&=e^{C_{t}}[|X_{t}|dC_{t}-dV^{s}_{t}+d(|Y_{t_{-}}+U_{t}|-|Y_{t_{-}}|)\nu_{t}+L_{t}^{Y}]+e^{C_{t}}[dM^{s}_{t}+d(|Y_{t_{-}}+U_{t}|-|Y_{t_{-}}|)\star\tilde{\mu}_{t}]\nonumber \\
\vspace{0.5cm}
&=e^{C_{t}}\left[dA_{t}-\frac{1}{2}d\langle M^{c}_{t}\rangle -j_{t}(\Delta M^{d}_{t})+d(|Y_{t_{-}}+U_{t}|-|Y_{t_{-}}|)\nu_{t}\right]+e^{C_{t}}\left[dM^{s}_{t}+d(|Y_{t_{-}}+U_{t}|-|Y_{t_{-}}|)\star\tilde{\mu}_{t}\right].\nonumber 
\end{align}
Thanks to structure condition of the semimartingale $Y$, the process $A$ is increasing. Notice that the martingale part of this last decomposition : $\bar{M}:=e^{C_{t}}dM^{s}_{t}+d(|Y_{t_{-}}+U_{t}|-|Y_{t_{-}}|)\star\tilde{\mu}_{t}$ have the following quadratic variation 
$d\langle \bar{M}\rangle_{t}=e^{2C_{t}}d\langle M^{c,s}\rangle_{t}+e^{2C_{t}}d\langle |Y_{t_{-}}+U_{t}|-|Y_{t_{-}}|)\star\tilde{\mu}_{t}\rangle$.
Adding and subtracting  $j_{t}(e^{C_{t}}\Delta M_{t}^{s,d})$ respectively $e^{2C_{t}}d\langle M^{s,c}\rangle$ to $\bar{X}$ yields to 
\begin{align*}
d\bar{X}^{\Lambda,C}_{t}&=d\tilde{A}_{t}-\frac{1}{2}d\langle e^{C_{t}}M^{s,c}_{t}\rangle -\frac{1}{2}j_{t}(e^{C_{t}}\Delta M^{s,d}_{t})+e^{C_{t}}d(|Y_{t_{-}}+U_{t}|-|Y_{t_{-}}|)\nu_{t}\\
&+e^{C_{t}}\left[dM^{s}_{t}+d(|Y_{t_{-}}+U_{t}|-|Y_{t_{-}}|)\star\tilde{\mu}_{t}\right].
\end{align*}
The process $\tilde{A}$ is increasing since $e^{C_{t}}j_{t}(\Delta M^{s,d}_{t})-j_{t}(e^{C_{t}}\Delta M^{s,d}_{t})\geq 0$. Furthermore $e^{C_{t}}d\langle M^{s,c}\rangle_{t}-e^{2C_{t}}d\langle M^{s,c}\rangle_{t}\geq 0$. 
Once again we add and subtract $ j_{t}(e^{C_{t}}\Delta M^{s,d})$ 
\begin{align*}
d\bar{X}^{\Lambda,C}_{t}&=d\bar{A}_{t}-\frac{1}{2}d\langle e^{C_{t}}M^{s,c}_{t}\rangle -\frac{1}{2}j_{t}(e^{C_{t}}\Delta M^{s,d}_{t})
+e^{C_{t}}\left[dM^{s}_{t}+d(|Y_{t_{-}}+U_{t}|-|Y_{t_{-}}|)\star\tilde{\mu}_{t}\right].
\end{align*}
The process $d\bar{A}=d\tilde{A}+d(j_{t}(e^{C_{t}}|Y_{t_{-}}+U_{t}|-|Y_{t_{-}}|)$ is increasing process.
This decomposition shows that $\bar{X}^{\Lambda,C}$ is a $Q$-submartingale.\\
Since  $\bar{X}^{\Lambda,C}$ is a $Q$-submartingale, it follows that $\exp(\bar{X}^{\Lambda,C})$ is a submartingale. Hence, for all stopping time $\sigma$, $\tau$ 
\begin{align*}
&\exp(\bar{X}_{\sigma}^{\Lambda,C})\leq \E\left[\exp(\bar{X}_{\tau}^{\Lambda,C})|\mathcal{F}_{\sigma}\right]\\
&\exp(e^{C_{\sigma}}|Y_{ \sigma}|+\int_{0}^{\sigma}e^{C_{\sigma,t}}d\Lambda_{t})\leq \E\left[\exp(e^{C_{\tau}}|Y_{ \tau}|+\int_{0}^{\tau}e^{C_{\tau,t}}d\Lambda_{t})|\mathcal{F}_{\sigma}\right].
\end{align*}
Taking $\int_{0}^{\sigma}e^{C_{\sigma,t}}d\Lambda_{t}$ in the right hand side we obtain 
\begin{align*}
&\exp(|Y_{ \sigma}|)\leq \E\left[\exp(e^{C_{\sigma,\tau}}|Y_{ \tau}|+\int_{\sigma}^{\tau}e^{C_{\tau,t}}d\Lambda_{t})\right].
\end{align*}
Hence we can write 
\begin{equation}
|Y_{ \sigma}|\leq \ln\E\left[\exp(e^{C_{\sigma,\tau}}|Y_{ \tau}|+\int_{\sigma}^{\tau}e^{C_{\tau,t}}d\Lambda_{t})\right].
\end{equation}
which ends the proof.
\end{proof}
\subsection{Quadratic variation and Stability result}
 \begin{Definition}\label{ClassQ}
 ${\cal S}_{Q}(|\xi|,\Lambda,C)$  is the class of all ${\cal Q}(\Lambda,C)$-semimartingales  $Y$ such that $$|Y_t|\le \bar\rho_t\left[e^{C_{t,T}}|\xi|+\int_t^T e^{C_{t,s}}d\Lambda_s\right], \quad a.s.$$
\end{Definition}
\begin{Proposition}\label{Stability} Let $(Y^{n})_{n}$ a sequence of ${\cal S}_Q(|\xi|,\Lambda,C)$ special semimartingales which canonical decomposition $Y^{n}=X^n_0-V^{n}+M^{n}$ which converge  in ${\cal H}^1$ to some process $Y$. Therefore the process $Y$ which canonical decomposition $Y=Y_0-V+M$ is an adapted c\`adl\`ag process which belongs to ${\cal S}_Q(|\xi|,\Lambda,C)$
such that:
\begin{equation*}
\label{limite:variation}
\lim_{ n \to \infty} \mathbb{E} \, \big[ \big(V^n - V\big)^* \big] = 0 \quad \mbox{and} \quad  \lim_{ n \to \infty} \|M^n - M\|_{\mathcal{H}^1} = 0.
\end{equation*}
\end{Proposition}
\begin{Remark}
The proof is built on the stability theorem of Barlow and Protter \cite{BP90}, based on uniform estimates of the quadratic variation part and the total variation of the semimartingale.\\
In \cite{BElK11}, the authors work in a continuous framework where they proved  that the class $S_{\Q}(|\xi|,\Lambda,C )$ is stable by a.s convergence. In the proposition above we prove that this class is also stable in the discontinuous setting.
\end{Remark}
\begin{proof} 

From proposition $(3.3)$ in \cite{ELKMN16}, the following exponential transformation  $(X^{\kappa})^{\Lambda,C}(Y)=Y^{n}+\Lambda+|Y^{n}|*C$ and $(X^{n})^{\Lambda,C}(-Y)$ are ${\cal Q}$-local submartingale. 
Then by the Yoeurp-Meyer decomposition,
 there exists an increasing process $A_{t}$ such that 
\begin{equation}
\begin{split}
&\exp((X^{n})^{\Lambda,C}_t(Y))=\exp(Y_0){\cal E}(\bar{M}_{t}+(e^{\bar{U}}-1).\tilde{\mu})\exp(\bar A_t).\\
&\exp((X^{\kappa})^{\Lambda,C}_t(-Y))=\exp(-Y_0){\cal E}((\bar{M}_{t}+(e^{\bar{U}}-1).\tilde{\mu}))\exp(\underline{A}_t).
\end{split}
\nonumber
\end{equation}

Therefore , for a stopping times $\sigma\leq T$, we obtain \\
$$\langle \bar{M}\rangle_{\sigma}=\int_{\sigma}^{T}\frac{d\langle\exp((X_t^{n})^{\Lambda,C}(Y))\rangle}{\exp(2(X_t^{n})^{\Lambda,C}(Y))}.$$
 $$\langle\underline{M}\rangle_{\sigma}=\int_{\sigma}^{T}\frac{d\langle\exp((X_{t}^{n})^{\Lambda,C}(-Y))\rangle}{\exp(2(X_{t}^{\kappa})^{\Lambda,C}(-Y))}.$$

Thank's to Garsia Lemma  \eqref{Garsia neveu} and It\^o formula we have for all $p\geq 1$
\begin{equation}\label{KLp}
\begin{split}
\mathbb{E}\left[\langle\exp(p(X_t^{n})^{\Lambda,C}(Y))\rangle_{T}\right]\leq C_{2}\quad \hbox{ and } \quad  \mathbb{E}\left[\langle\exp(p(X_t^{n})^{\Lambda,C}(-Y))\rangle_{T}\right]\leq C_{1}.\quad 
\end{split}
\end{equation}
Then the estimates of $\langle \bar{M}\rangle$ and $\langle \bar{M}\rangle$ comes from Cauchy Schwartz inequality 
\begin{equation*}
\mathbb{E}\left[\langle\bar M^p_T\rangle\right]\leq C \quad \hbox{ and } \mathbb{E}\left[\langle\underline M\rangle^p_T\right]\leq C.
\end{equation*}
In the other hand , applying It\^o formula and using the fact that 
$$2[(e^{\delta U}-\delta U-1)+(e^{-\delta U}+\delta U-1)]\le |e^{\delta U}-1|^2+|e^{-\delta U}-1|^2.$$ leads to 
$$\mathbb{E}\left[\int_0^T|dV^n_s|\right]\le \mathbb{E}\left[\int_0^T {1\over 2}d\langle {(M^c)}^\kappa\rangle_s+\int_E[j(U^\kappa(s,x))+j(-U^\kappa(s,x))]\nu(ds,dx)\right]\le C^{'}.$$
Finally Applying the Barlow-Protter stability theorem \cite{BP90}, we obtain that the limit $Y$ of $Y^{n}$ is in fact a special semimartingale with the canonical decomposition $Y:=Y_{0}-V+M$
 satisfying:
\begin{equation}
\label{estimate:variation}
\mathbb{E} \Big[\int_0^T |dV_s| \, \Big] \leq C, \quad \mbox{and} \quad \mathbb{E} \, \big[ \big(M \big)^* \big] \leq C.
\end{equation}
and we have
\begin{equation}
\label{limite:variation}
\lim_{ n \to \infty} \mathbb{E} \, \big[ \big(V^n - V\big)^* \big] = 0 \quad \mbox{and} \quad  \lim_{n \to \infty} \|M^n - M\|_{\mathcal{H}^1} = 0.
\end{equation} 
Moreover, the semimartingale $X$ is also a $S_{\Q}(|\xi|,\Lambda,C )$ since a.s
\begin{equation}
|Y^n_\sigma|\le \bar\rho_\sigma\left[e^{C_{\sigma,T}}|\eta_T|+\int_\sigma^T e^{C_{\sigma,s}}d\Lambda_s\right] .
\end{equation}
Passing to the limit when $n$ goes to $\infty$.  We finally obtain 

\begin{equation}
|X_\sigma|\le \bar\rho_\sigma\left[e^{C_{\sigma,T}}|\xi|+\int_\sigma^T e^{C_{\sigma,s}}d\Lambda_s\right] .
\end{equation}
\end{proof}
In order to prove the existence of the solution we need the following lemma. Here we justify the existence of  solution of the BSDE associated to $(\bar{q}^{n,\kappa},|\xi|)$ and $(\underline{q}^{m,\kappa},-|\xi|)$.
\begin{Lemma}\cite{ELKMN16}
We consider the following approximation 
\begin{equation}
\bar{q}^{n,\kappa}=:=\bar{q}\wedge b(y,z,u)=\displaystyle{\inf_{r,w,v}}\left\lbrace \bar{q}^{\kappa}(r,w,v)+n|y-r|+n|z-w|+n|u-v|_{\nu}\right\rbrace .
\end{equation}
\begin{equation}
\underline{q}^{m,\kappa}:=\bar{q}\vee b(y,z,u)=\displaystyle{\sup_{r,w,v}}\left\lbrace \underline{q}^{\kappa}(r,w,v)+m|y-r|+m|z-w|+m|u-v|_{\nu}\right\rbrace .
\end{equation}
where $\bar{q}^{\kappa}(r,w,v)=|l|_{t}+c_{t}|r|+\frac{1}{2}|w|^{2}+\frac{1}{\delta}j(\delta u)$ and $\underline{q}^{\kappa}(r,w,v)=-|l|_{t}-c_{t}|r|-\frac{1}{2}|w|^{2}-\frac{1}{\delta}j(-\delta u).$\\
We have \\
$\bullet$ The coefficient $(\bar{q}^{n,\kappa})_{n,\kappa}$ (resp.$\underline{q}^{n,\kappa}$ ) satisfy the structure condition in Assumptions \eqref{H1} and  converges monotonically to $\bar{q}$ in $(n,\kappa)$. \\
Moreover,  there exists respectively a unique solution $(\mathcal{\bar{Y}}^{n,\kappa},\mathcal{\bar{Z}}^{n,\kappa},\mathcal{\bar{U}}^{n,\kappa})$,$(\mathcal{\underline{Y}}^{m,\kappa}, \mathcal{\underline{Z}}^{m,\kappa},\mathcal{\underline{U}}^{m,\kappa})$ for the BSDEJ associated respectively to  $(\bar{q}^{n,\kappa},\xi)$ and $(\underline{q}^{m,\kappa},-\xi)$ satisfying \\
\begin{equation}\label{P}
|\mathcal{\bar{Y}}^{n,\kappa}_{t}|\leq \bar{\rho}_{t}\left[ e^{C_{t,T}}|\xi|+\int_{t}^{T}e^{C_{t,s}}d\Lambda_{s}\right] \hspace{2cm } \forall t\in[0,T], \P-a.s
\end{equation}
\end{Lemma}
\begin{Lemma}(Garsia-Neveu)\label{Garsia neveu}\\
Let $A$ be a predictable c\`adl\`ag increasing process and $\Phi$ a random variable, positive integrable. For any stopping times $\sigma\leq T$, we have $\E\left[A_{T}-A_{\sigma}|\mathcal{F}_{\sigma}\right]\leq \E\left[ U.1_{\sigma < T}|\mathcal{F}_{\sigma}\right]$. \\Then for all $p\geq 1$, $$ \E[A_{T}^{p}]\leq p^{p}\E[U^{p}].$$
\end{Lemma}
\bibliographystyle{plain}
\bibliography{bibfileEQBSDEJ}
\end{document}